\documentclass[a4paper,UKenglish,cleveref, autoref, thm-restate]{lipics-v2021}

\usepackage{graphicx}
\usepackage{gensymb}
\usepackage{mathtools}
\usepackage{mathrsfs}
\usepackage{wrapfig}
\usepackage{hyperref}
\usepackage{mathtools}
\usepackage{textcomp}
\usepackage{pdfpages}
\usepackage{pgfplots}
\usepackage{amsmath}
\usepackage[utf8]{inputenc}
\usepackage{t1enc}
\usepackage{amssymb}
\usepackage{amsthm}
\usepackage{verbatim}
\usepackage{graphicx}
\usepackage{xr}
\usepackage{subcaption}
\usepackage{caption}
\usepackage{float}

 \newtheorem*{definition*}{Definition}
\newtheorem*{observation*}{Observation}
\newtheorem{conj}{Conjecture}[]
 \newtheorem{problem}[conj]{Problem}
\newtheorem*{lemma*}{Lemma}
\hideLIPIcs
 
\title{Saturated Partial Embeddings of Maximal Planar Graphs} %TODO Please add

\author{Alexander Clifton}{Center for Mathematical Modeling, University of Chile, Santiago, Chile}{yoa@ibs.re.kr}{}{Research supported in part by the Institute for Basic Science (IBS-R029-C1)}

\author{Dániel G. Simon}{HUN-REN Alfréd Rényi Institute of Mathematics, Budapest, Hungary \and Eötvös Loránd University, Budapest, Hungary} {dgsimon@renyi.hu}{https://orcid.org/0000-0002-3750-9666}{Research supported by ERC advanced grant no. 882971: Geoscape.}

\authorrunning{A. Clifton, D. G. Simon} %TODO mandatory. First: Use abbreviated first/middle names. Second (only in severe cases): Use first author plus 'et al.'

\Copyright{Alexander Clifton, Dániel G. Simon} %TODO mandatory, please use full first names. LIPIcs license is "CC-BY";  http://creativecommons.org/licenses/by/3.0/

 %TODO mandatory: Please choose ACM 2012 classifications from https://dl.acm.org/ccs/ccs_flat.cfm 

 %TODO mandatory; please add comma-separated list of keywords

\category{} %optional, e.g. invited paper

\relatedversion{} %optional, e.g. full version hosted on arXiv, HAL, or other respository/website
\acknowledgements{We would like to thank Nika Salia for proposing the question.}%optional

\nolinenumbers %uncomment to disable line numbering

\begin{document}

\maketitle
\begin{abstract}
We investigate two notions of saturation for partial planar embeddings of maximal planar graphs. Let $G = (V, E) $ be a vertex-labeled maximal planar graph  on $ n $ vertices, which by definition has $3n - 6$ edges. We say that a labeled plane graph $H = (V, E')$ with $E' \subseteq E$ is a \emph{labeled plane-saturated subgraph} of $G$ if no edge in $E \setminus E'$ can be added to $H$ in a manner that preserves vertex labels, without introducing a crossing. The \emph{labeled plane-saturation ratio} $lpsr(G)$ is defined as the minimum value of $\frac{e(H)}{e(G)}$ over all such $H$. We establish almost tight bounds for $lpsr(G)$, showing $lpsr(G) \leq \frac{n+7}{3n-6}$ for $n \geq 47$, and constructing a maximal planar graph $G$ with $lpsr(G) \geq \frac{n+2}{3n-6}$ for each $n\ge 5$.

Dropping vertex labels, a \emph{plane-saturated subgraph} is defined as a plane subgraph $H\subseteq G$ where adding any additional edge to the drawing either introduces a crossing or causes the resulting graph to no longer be a subgraph of $G$. The \emph{plane-saturation ratio} $psr(G)$ is defined as the minimum value of $\frac{E(H)}{E(G)}$ over all such $H$. For all sufficiently large $n$, we demonstrate the existence of a maximal planar graph $G$ with $psr(G) \geq \frac{\frac{3}{2}n - 3}{3n - 6} = \frac{1}{2}$.
%We investigate saturation properties of partial embeddings in maximal planar graphs. Let \( G = (V, E) \) be a maximal planar graph  on \( n \) vertices that has \( 3n - 6 \) edges by definition. For a plane drawing of a subgraph \( H = (V, E') \) with \( E' \subseteq E \), we define \( H \) is a \emph{labeled plane-saturated subgraph} of \( G \) if no edge in \( E \setminus E' \) can be added to the drawing of \( H \) without introducing a crossing while preserving vertex labels. The \emph{labeled plane-saturation ratio} \( lpsr(G) \) is the minimum \( \frac{e(H)}{e(G)} \) over all such \( H \). We establish almost tight bounds for \( lpsr(G) \), showing \( lpsr(G) \leq \frac{n+5}{3n-6} \) for \( n \geq 47 \), and construct a graph with \( lpsr(G) \geq \frac{n+2}{3n-6} \).

%Dropping vertex labeling, we define \emph{plane-saturated subgraphs} where no edge or vertex can be added without crossings or breaking subgraph inclusion. The \emph{plane-saturation ratio} \( psr(G) \) is similarly defined. We demonstrate the existence of a maximal planar graph \( G \) with \( psr(G) \geq \frac{\frac{3}{2}n - 3}{3n - 6} = \frac{1}{2} \).
    
\end{abstract}

\section{Introduction}

A graph is called planar, if it can be embedded in $\mathbb{R}^2$ such that each vertex is mapped to a different point, and each edge is mapped to a simple topological curve, where no two curves may intersect, overlap or touch each other except by possibly sharing their common endpoints. A plane graph is a planar graph together with a specific plane-embedding.

This mapping is not always straightforward. When drawing the edges of the planar graph one by one, at some point it can become impossible to draw new edges of the graph without creating a crossing. The simplest example of this phenomenon is the graph which consists of $5$ vertices, out of which $3$ form a triangle, and the other $2$ form an edge. If the triangle is drawn first and the endpoints of the extra edge are placed inside and outside the triangle, respectively, it becomes impossible to complete the drawing without introducing a crossing. See Figure \ref{fig:5csucs}. We will call this drawing $H$ a \textbf{plane-saturated subgraph} of $G$, properly defined later. Our work focuses on identifying the optimal maximal planar graphs in which any saturated subdrawing retains a large number of edges, ensuring that the graph cannot easily be drawn in a ''bad'' non-planar manner.

Saturation problems have been studied by mathematicians since the mid-20th century. For any graph $H$, the saturation number $sat(n,H)$ denotes the minimum number of edges in an $n$-vertex graph that contains no subgraph isomorphic to $H$, but for which adding any extra edge creates a copy of $H$. The topic was first studied by Zykov \cite{zykov} in 1949, and in 1964, the exact value of $sat(n,K_t)$ was shown by Erdős, Hajnal and Moon \cite{erdos}. Despite significant progress, many open questions remain. For example, if $C_l$ denotes a cycle of length $l$, the exact value of $sat(n,C_l)$ is known only for $l\le5$. Specifically, case $l=3$ follows from \cite{erdos}; $l=4$ was resolved by Ollmann in 1972 \cite{ollmann}; and $l=5$ was addressed with an exact upper bound by Fisher, Fraughnaugh, and Langley in 1995 \cite{fisher} and a matching lower bound by Chen in two subsequent papers from 2009 and 2011 \cite{chen1, chen2}. For larger values of $l$, Füredi and Kim conjectured in \cite{furedi} that their upper bound $sat(n,C_l)=(1+\frac{1}{l-4})n+O(l^2)$ is exact, but so far this conjecture remains open.

For general graphs $H$, the best known upper bound on $sat(n,H)$ was established by Faudree and Gould in 2013 \cite{faudree2}, while the first general lower bound was provided by Cameron and Puleo in 2022 \cite{cameron}. For a comprehensive overview of saturation problems, the survey by Faudree, Faudree, and Schmitt \cite{faudree} is recommended.

 Our research on saturation problems in plane drawings of graphs was heavily inspired by the recent work of Clifton and Salia from 2024 \cite{clifton}. Their research, in turn, was motivated by earlier studies on saturation problems for simple topological drawings of graphs, conducted by Kynčl, Pach, Radoičič, and Tóth \cite{kyncl}, as well as others \cite{brass, fulek, pach, suk}. This area has been attracting increasing attention; before submission, another research group \cite{barat} informed us that they are also working on the plane-saturation problem addressed in this paper.

\begin{figure}
    \centering
    \includegraphics[scale=0.8]{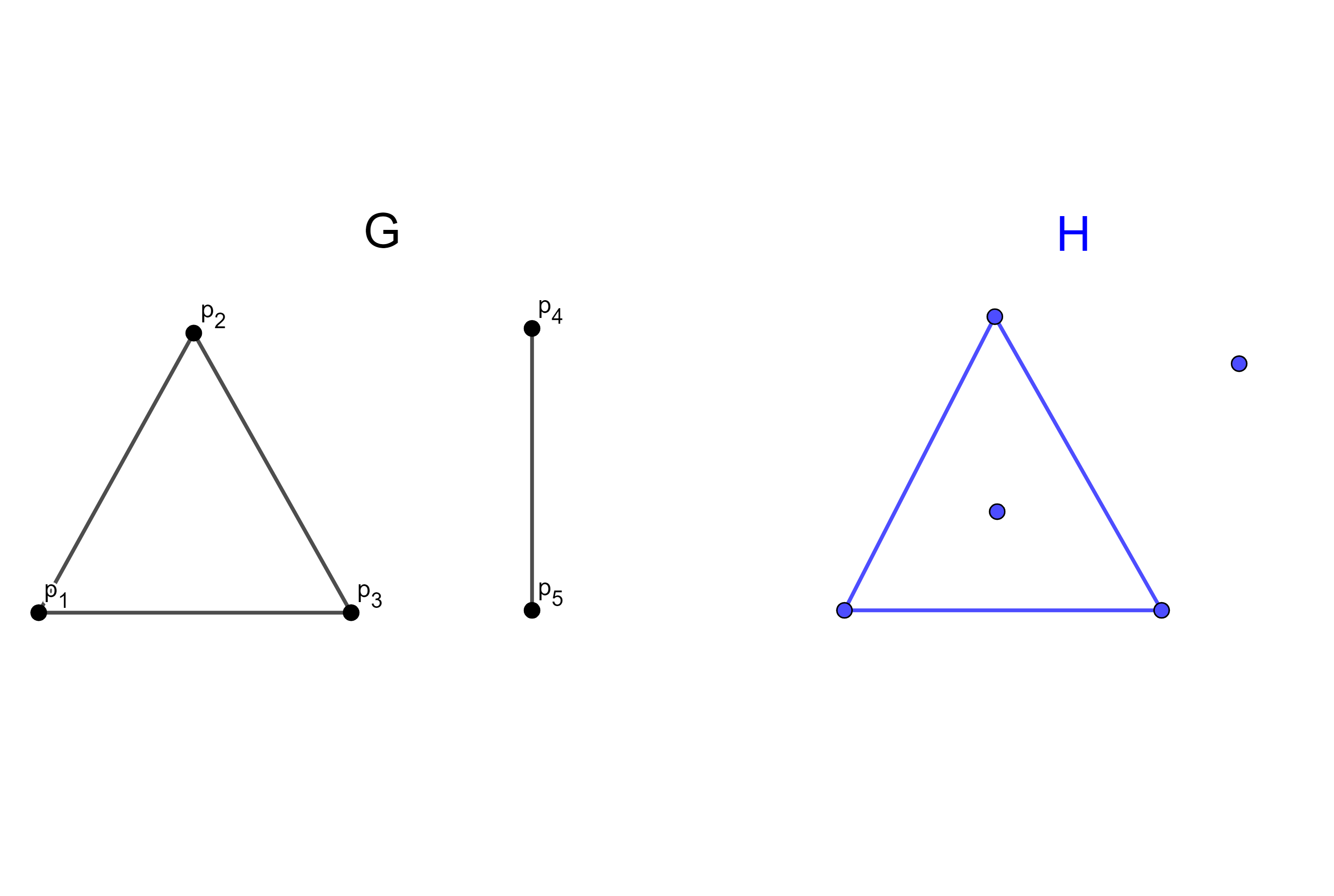}
    \caption{For the graph G on the left side in black, the plane graph H on the right in blue is a plane-saturated subgraph of G. If we add labels to the blue vertices, it would also be a labeled plane-saturated subgraph of G.}
    \label{fig:5csucs}
\end{figure}

Throughout the paper, we let $G=(V,E)$ be a maximal planar graph on $n$ vertices and $3n-6$ edges. We introduce two properties of $G$: the \emph{labeled plane-saturation ratio}, denoted by $lpsr(G)$, and the \emph{plane-saturation ratio}, denoted by $psr(G)$. The plane-saturation ratio was first introduced by Clifton and Salia in \cite{clifton}, while the labeled plane-saturation ratio is a new definition.

\begin{definition*}
    A plane graph $H=(V',E')$ with labeled vertices is called a \textbf{labeled plane-saturated subgraph} of $G=(V,E)$ if: \begin{itemize}
        \item $V'=V$,
        \item $E'\subseteq E$,
        \item adding any edge to $H$ either introduces a crossing in the drawing, or results in $E'$ no longer being a subset of $E$.
    \end{itemize}
     The \textbf{labeled plane-saturation ratio} of a graph $G$, denoted by $lpsr(G)$, is defined as:
     
     \centering
    $lpsr(G)=\min\{\frac{e(H)}{e(G)}\mid$$ H$ is a labeled plane-saturated subgraph of G$\}$,\\
    where $e(H)$ and $e(G)$ denote the number of edges in $H$ and $G$, respectively. 
\end{definition*}.

\begin{definition*}
    A plane graph $H$ is called a \textbf{plane-saturated subgraph} of $G$ if:
    \begin{itemize}
        \item $H$ is a subgraph of $G$ with the same number of vertices,
        \item adding any edge to $H$ either introduces a crossing, or results in $H$ no longer being a subgraph of $G$.
    \end{itemize} 
     The \textbf{plane-saturation ratio} of a graph $G$, denoted by $psr(G)$, is defined as: \\ \centering
    $psr(G)=\min\{\frac{e(H)}{e(G)}\mid$$ H$ is a plane-saturated subgraph of G$\}$,\\
    where $e(H)$ and $e(G)$ denote the number of edges in $H$ and $G$, respectively. 
\end{definition*}.

At first it might be difficult to see the subtle difference between the two definitions. It is easy to see that every plane-saturated subgraph $H$ of $G$ can be labeled to give a labeled plane-saturated subgraph of $G$. Since $H$ is a subgraph of $G$, by definition there exists at least one valid vertex labeling of $H$ such that $E(H)\subseteq E(G)$. Moreover, by definition of plane-saturation, no new edges can be added to the drawing of $H$ without adding a crossing or violating the property that the graph is a subgraph of $G$, thus making the labeled $H$ labeled plane-saturated.

On the other hand, there exist labeled plane-saturated subgraphs that are not plane-saturated. An example illustrating the difference is shown in Figure \ref{fig:not_plane-sat}.

The following observation is a simple consequence of the previous paragraphs.
\begin{observation*}
    For any planar graph $G$, $lpsr(G)\leq psr(G)$.
\end{observation*}

\begin{figure}
    \centering
    \includegraphics[scale=0.8]{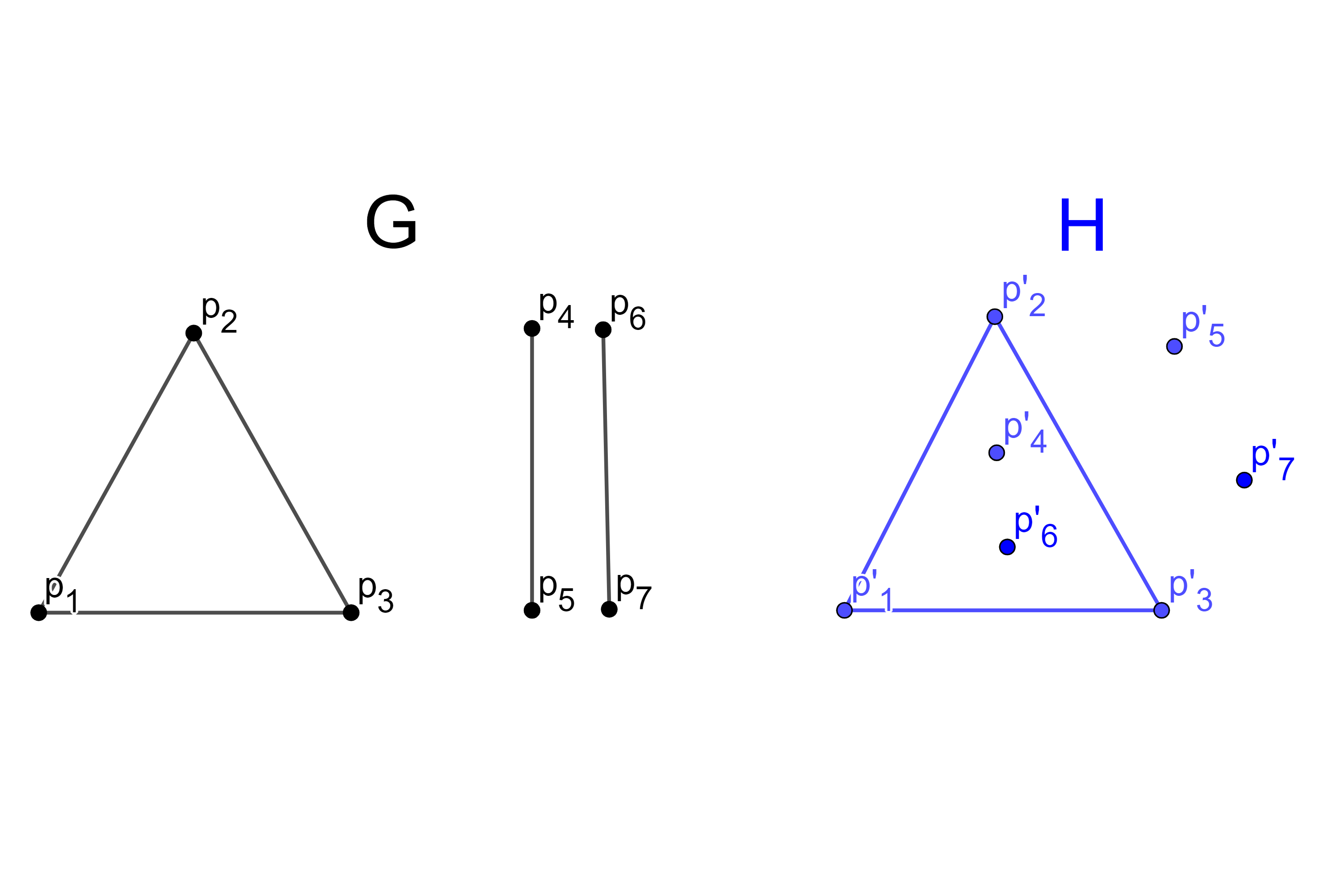}
    \caption{For the graph G on the left side in black, the plane graph H on the right in blue is a labeled plane-saturated subgraph of G. If we remove the labels of the blue vertices, it is \textbf{not} a plane-saturated subgraph of G, as we can switch the roles of $p'_5$ and $p'_6$ to add 2 more edges.}
    \label{fig:not_plane-sat}
\end{figure}

We investigate the values of $lpsr(G)$ and $psr(G)$ for maximal planar graphs $G$. The main contributions of this paper are summarized in the following theorems. Below, we present our first two results concerning the labeled plane-saturation ratio:

\begin{theorem}\label{lower}
For $n\ge 5$, there exists a maximal planar graph $G$ with $n$ vertices such that $lpsr(G)\geq \frac{n+2}{3n-6}$.%, \\ where $lpsr(G)$ is the labeled plane-saturation ratio of the graph $G$.
\end{theorem}

\begin{theorem}\label{upper}
    Let $G$ be a maximal planar graph $G$ on $n\geq47$ vertices. Then,
  $lpsr(G)\leq \frac{n+7}{3n-6}$.%, \\ where $lpsr(G)$ is the labeled plane-saturation ratio of the graph $G$.
\end{theorem}

This means that up to an additive constant of $3$, we have determined the maximal size of a labeled plane-saturated subgraph of a maximal planar graph on $n$ vertices. For $n\leq 46$, our proof of Theorem~\ref{upper} does not work. If $n$ tends to infinity, there is an immediate corollary of the above theorems:

\begin{corollary}
    The minimum value of $lpsr(G)$ among all maximal planar graphs $G$ with $n$ vertices converges to $\frac13$ as $n$ goes to infinity.
\end{corollary}
For the plane-saturation ratio, we found a stronger lower bound:

\begin{theorem}\label{lowerpsr}
For all $n$, there exists a maximal planar graph $G$ on $n$ vertices such that $psr(G)\geq \frac{\frac32n-3}{3n-6}=\frac12$.%, \\ where $psr(G)$ is the plane-saturation ratio of the graph $G$.
\end{theorem}

We also give some upper bounds on the minimal size of a plane-saturated subgraph of $G$. When $G$ has a vertex of degree $\Omega(n)$, these results yield a linear improvement over $3n-6$.

\begin{theorem}\label{psrupper}
If the maximum degree vertex of maximal planar graph $G$ on $n$ vertices has degree $d$, then $G$ has a plane-saturated subgraph with at most $$3n-6-min(d,n-d-1)$$ edges. Moreover, if $n\geq75$ and $d\geq 0.933n$, then $G$ has a plane-saturated subgraph with at most $$(3-\frac{d^2}{12n^2+dn})n+1$$ edges. 
\end{theorem}

In Section \ref{sec2}, we prove Theorem \ref{lower}. In Section \ref{sec3}, we prove Theorem~\ref{upper}, showing the upper bound on the labeled plane-saturation ratio. Finally, in Section \ref{sec4}, we prove Theorems \ref{lowerpsr} and \ref{psrupper}. The paper concludes with some remarks and open questions.

\section{Construction for the Lower Bound on $lpsr(G)$}\label{sec2}
 
 Define $G=(V,E)$ to be the following graph on $n\ge 5$ vertices: $V=\{v_1, v_2, ..., v_n\}$, $E= \{(v_1,v_i) \mid 3\leq i\leq n\}\cup\{(v_2,v_j) \mid 3\leq j\leq n\}\cup\{ (v_i,v_{i+1}) \mid 3\leq i\leq n-1\}\cup \{(v_3,v_n)\}$. It is a maximal planar graph as can be seen in Figure \ref{construction}.

We prove the following:
\begin{theorem} \label{sec2main}
    Any labeled plane-saturated subgraph $H$ of $G$ has an edge-set of size at least $n+2$. That is, $lpsr(G)\ge \frac{n+2}{3n-6}$.
\end{theorem}

\begin{proof}
Let $H$ be a labeled plane-saturated subgraph of $G$. First we prove the theorem for $n\geq6$. 

First note that in order to be (labeled) plane-saturated, some cycle must be present in $H$, since if the drawing was a forest, any new edge could be added to it without introducing a crossing. 

If the cycle $v_3v_4\dots v_nv_3$ is contained in $H$, then every vertex on the cycle must be connected to at least one of $v_1$ and $v_2$. This is because for any $3\leq i\leq n$, $v_i$ lies on the boundary of at least two faces. Since $v_1,v_2$ are added either inside or outside the $(n-2)$-cycle, at least one of the two faces whose boundary includes $v_i$ contains at least one of $v_1, v_2$. Without loss of generality, say one face contains $v_1$. Then the edge $(v_1,v_i)$ must be present in $H$, as otherwise we could connect $v_1$ and $v_i$ through the face without introducing a crossing, contradicting the fact that $H$ is saturated. The cycle of length $n-2$ and the at least $n-2$ additional edges from the vertices on the cycle to $\{v_1,v_2\}$ give at least $2n-4$ edges in $H$, which is larger than or equal to $n+2$ for $n\geq 6$.

 \begin{figure}
	\begin{center}
		\includegraphics[scale=0.8]{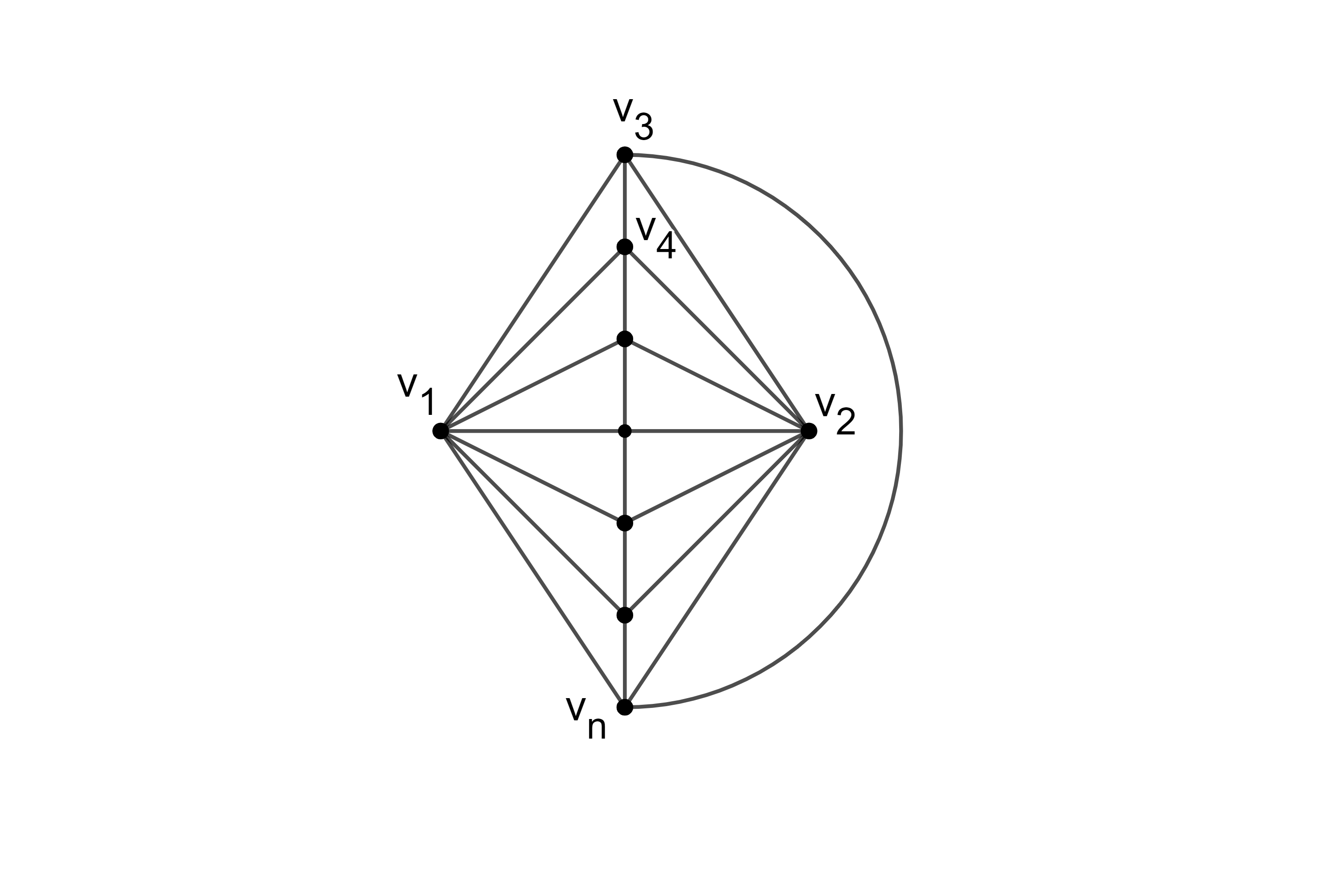}
		\caption{$G$ is a maximal planar graph.}
		\label{construction}
	\end{center}
\end{figure}

Now we prove some lemmas. They are not numbered, as they are only intended to make the proof easier to follow.

\begin{lemma*}
    If in $H$, each cycle contains at least one of $v_1,v_2$, then for each $i\geq3$, at least one of $(v_1,v_i), (v_2,v_i)$ is an edge of $H$.
\end{lemma*}
 \begin{proof}
     Assume that in $H$, every cycle contains $v_1$ or $v_2$. Suppose that there is some $v_i$ with $i\ge 3$ such that there is no edge from $v_i$ to $v_1$ nor to $v_2$. $v_i$ must lie within some face, or on the boundary of a face. The boundary of each face contains a cycle of length at least $3$. Therefore the boundary of the face around $v_i$ includes either $v_1$ or $v_2$ so it is possible to add an edge from either $v_1$ or $v_2$ to $v_i$, without introducing a crossing. That means $H$ cannot be saturated, contradicting our assumption that neither $(v_1,v_i)$ nor $(v_2,v_i)$ was present.
     \end{proof}
     As $v_3v_4\dots v_n$ is the only cycle of $G$ containing neither $v_1$ nor $v_2$, we may assume that each cycle of $H$ contains at least one of $v_1, v_2$. Thus, for each $i\geq3$, $v_i$ is connected to at least one of $v_1,v_2$ via an edge in $H$. Thus, the subgraph $H$ has at most two connected components, where two is only possible if $v_1, v_2$ are in separate connected components.

\begin{lemma*}
    $H$ is connected.
\end{lemma*}
\begin{proof}
If $v_1,v_2$ are in separate components $H_1, H_2$ where $H_1$ is a tree, then $H_2$ has a cycle (any saturated drawing has at least one cycle) and we can always add an edge from $v_1$ to a boundary vertex of whatever face of $H_2$ contains $H_1$ in its interior. Thus, we may assume that both $H_1$ and $H_2$ contain a cycle. Then, $H_2$ is contained in some face of $H_1$. That face necessarily contains $v_1$ as a boundary vertex since any face boundary contains a cycle, and each cycle contains $v_1$ or $v_2$. Then it is possible to add an edge between $v_1$ and a vertex $v_i$ for $i\ge3$ on the outer face of $H_2$.  Thus, $v_1, v_2$ could not have been in separate components, showing $H$ is connected.
\end{proof}

From now on, we can assume $H$ is connected, with each cycle containing at least one of $v_1,v_2$.

\begin{lemma*}
    $v_1,v_2$ must each lie on a cycle in $H$.
\end{lemma*}

\begin{proof}
    If we assume the lemma to be false, then without loss of generality, no cycle contains $v_1$. The drawing of each face consists of a cycle around it, plus possibly trees rooted at each vertex, as seen in Figure \ref{fig:1-cycle}. Thus, $v_1$ lies on such a tree. Then it is possible to add an edge from $v_1$ to any vertex (except $v_2$) on its surrounding face it is not already adjacent to, without causing a crossing. Hence all these edges must be present in $H$ since it is a saturated subgraph. The boundary of each face contains a cycle of length at least $3$, so the one around $v_1$ contains at least two vertices, $u, w$ that are not $v_1$ or $v_2$. Since $v_1$ must be adjacent to all such vertices, the path from $u$ to $w$ along with the edges $v_1u$ and $v_1w$ forms a cycle containing $v_1$, giving a contradiction. Thus $v_1$ and $v_2$ are each contained in at least one cycle.   %Since $v_1$ must be connected to all such vertices, it introduces a smaller cycle, so if $H$ is saturated, both $v_1$ and $v_2$ are contained in at least one cycle. 
    \end{proof}

\begin{figure}
    \centering
    \includegraphics[scale=0.6]{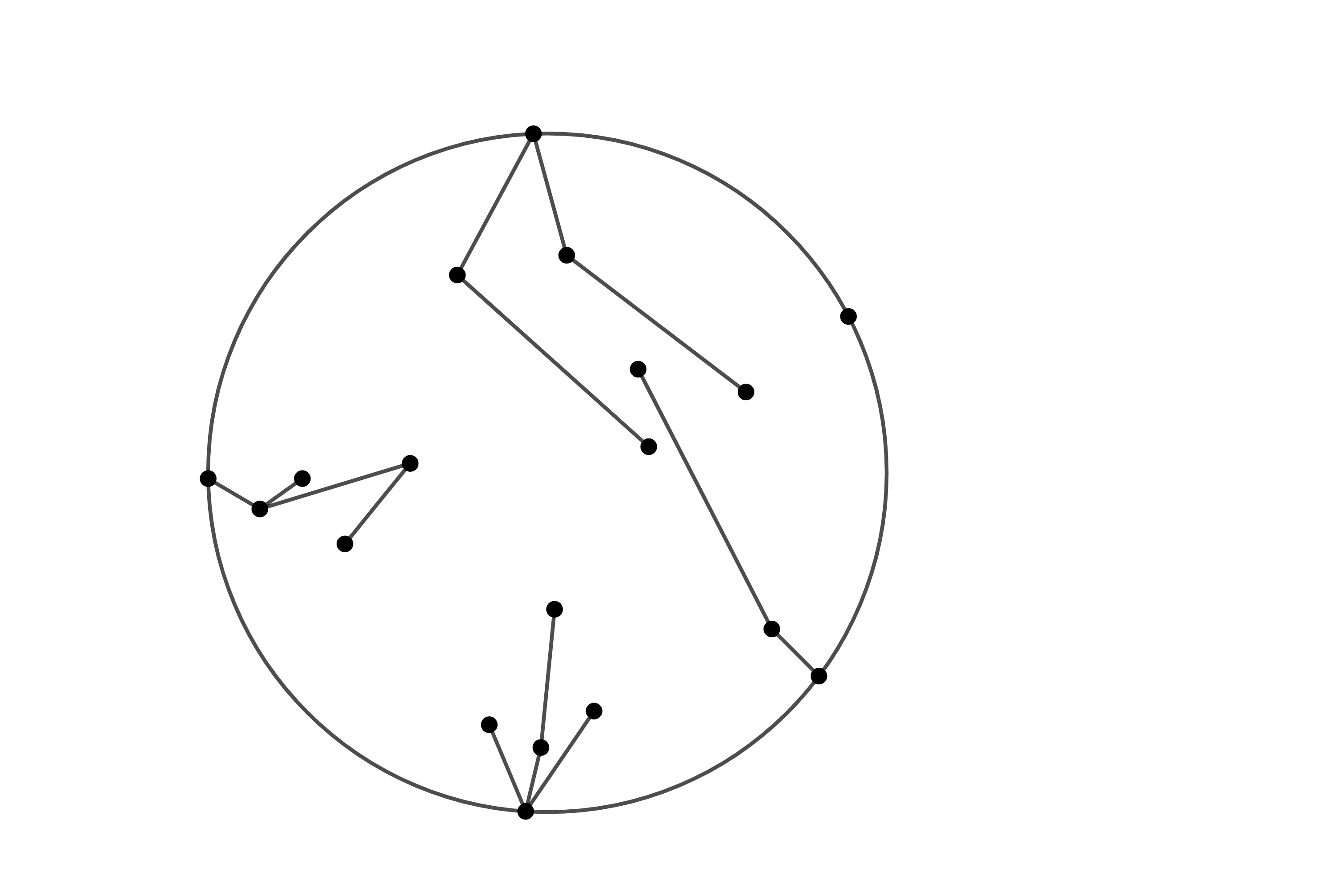}
    \caption{If $H$ is connected, the drawing of each face looks like a cycle, together with trees potentially attached to each vertex.}
    \label{fig:1-cycle}
\end{figure}

Using the above lemma, and that $H$ is saturated, we can prove the following:

\begin{lemma*}
    $H$ has at least three cycles. %A saturated drawing of <---I removed this cause our H is now a plane graph
\end{lemma*}

\begin{proof}
$H$ cannot have only one cycle through both $v_1$ and $v_2$, since then every other vertex is either on the cycle, or lies inside one of the two faces bounded by the cycle. Any $v_i$ not lying on the cycle cannot be adjacent to both $v_1, v_2$, as that would form a second cycle. Without loss of generality, such a $v_i$ is nonadjacent to $v_2$, but then we can add edge $v_2v_i$ through the face without causing a crossing, contradicting that $H$ is saturated. Thus, the cycle contains all $n$ vertices, and $H$ is simply $C_n$. As $n>4$, there is some vertex $v_i$ with $i\ge 3$ on the cycle, but not adjacent to $v_1$. Then edge $v_1v_i$ can be added through the face,  again contradicting that $H$ was saturated.

Thus, a saturated drawing has at least two cycles, so at least three faces. If there are exactly two cycles, and both contain $v_1$ and $v_2$, that would yield a third cycle. If instead one of $v_1, v_2$, say $v_1$ is only on one cycle, then the other cycle contains at least two vertices $u,w$ besides $v_2$ and we may assume at least one of these is non-adjacent to $v_1$ in $H$, as otherwise we have two cycles sharing a path from $u$ to $w$, and thus at least three cycles. It is then possible to add an edge between $v_1$ and a non-$v_2$ vertex on the other cycle, without introducing a crossing, contradicting that $H$ is saturated. Thus, a labeled plane-saturated $H$ has at least three cycles.
\end{proof}

\begin{figure}
    \centering
    \includegraphics[scale=0.8]{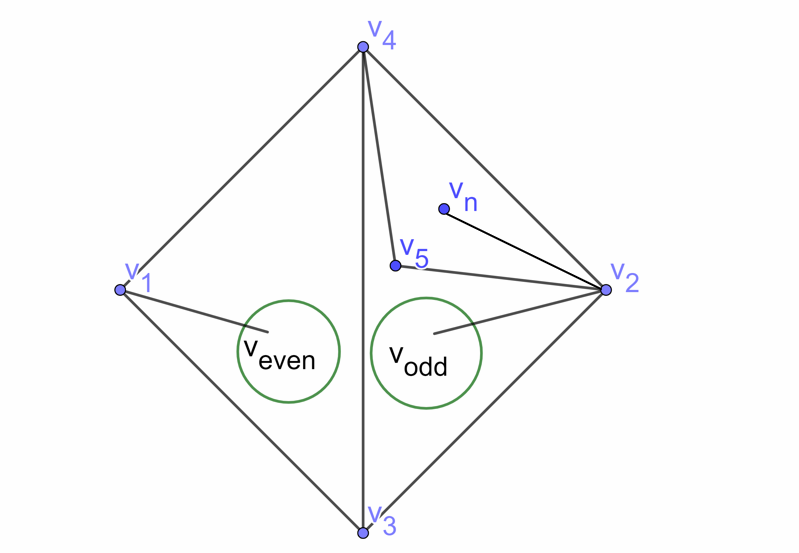}
    \caption{Saturated drawing of $n+2$ edges for Proposition \ref{tight1}. In the green circles we placed all the remaining even and odd indexed vertices, respectively. They are only connected to $v_1$ and $v_2$, respectively.}
    
    \label{fig:n+2}
\end{figure}

Now we return to the proof of Theorem \ref{sec2main}.

If $H$ has at least four faces, then by Euler's criterion, the number of edges is at least $n+2$. If there are only three faces, then the graph consists of two cycles which share a path, plus possibly some additional trees rooted at these vertices. Let $P$ be the shared path and $C_1, C_2$ be the cycles sharing it. Without loss of generality, suppose $v_1$ lies on $C_1$, and $v_2$ lies on $C_2$. (Note that it is possible for $v_1$ or $v_2$ to lie on both.) $v_1$ can be connected to each non-$v_2$ vertex on $C_1$ through the face, which means that since $H$ is saturated, $C_1$ can contain at most $2$ vertices apart from $v_1,v_2$. Similarly, $C_2$ can only contain at most $2$ extra vertices. Each face is bounded by a cycle, so any vertex inside a face can have an edge added between it and one of $v_1,v_2$ on the boundary of that face, without introducing a crossing. Thus any vertex not contained on $C_1$ or $C_2$ is already a neighbour of $v_1$ or $v_2$, and consequently must have degree $1$. 

Suppose some endpoint of $P$ is $v_i$ for $i\ge 3$. That $v_i$ must be adjacent to $v_1$ along one of $C_1, C_2$ and to $v_2$ along the other. As $v_i$ only has three neighbours along $C_1$ and $C_2$ and no others, there must be some vertex $v_j$ where $j\ge 3$ and $(v_i,v_j)$ is an edge of $G$ such that $(v_i,v_j)$ is not an edge of $H$. However, $v_i$ is incident to every face of $H$, so it is possible to add the edge $(v_i,v_j)$ without introducing a crossing, contradicting that $H$ is saturated. Therefore, the endpoints of $P$ must be $v_1,v_2$ but as they are not adjacent in $G$, they must have at least one vertex in between them along $P$. In fact, they can only have one vertex in between them along $P$ since $C_1, C_2$ have at most four edges each and $v_1,v_2$ also need to have a vertex between them on each of $C_1,C_2$ that is not along $P$. As both $v_1,v_2$ lie on the outer face, it is impossible for a degree $1$ vertex to be on the boundary of the outer face since then it would be possible to add an edge from it to whichever of $v_1,v_2$ it is not already adjacent to, without causing a crossing. Thus any degree $1$ vertex is on a face incident to the vertex $w$, where $w$ is the common neighbour of $v_1,v_2$ along $P$. This means the two vertices among $\{v_3,\dots,v_n\}$ which are neighbours of $w$ in $G$ necessarily lie on a face incident to $w$ so it is possible to add an edge from one to $w$ without causing a crossing, contradicting that $H$ is saturated.
Thus, there are no degree $1$ vertices and $C_1,C_2$ contain all the vertices of the graph. But then $H$ contains at most $6$ vertices, as each cycle has at most $2$ vertices that are not $v_1$ or $v_2$. It can only have $6$ vertices if none of the extra vertices lie on $P$. However, then $P$ must contain only $v_1$ and $v_2$, meaning there is an edge between $v_1$ and $v_2$, a contradiction. %But then the third (outer) face has all of $v_3,v_4,v_5,v_6$ on its boundary, hence at least one of $v_1,v_2$ can be connected to all of them, which is not allowed.

The proof is done for $n\geq6$. %For $n=1,2,3,4,5$ we need to check by hand. For $n=1,2,3,4$ it is trivial, since then the only saturated subgraphs are the whole graph.  %I removed 1 to 4 because the graph G isn't defined for those. Though for n=4, we can still get n+2
For $n=5$, we can only separate $2$ vertices from each other, hence out of the $9$ edges, at least $8$ are present in the saturated subdrawing, which is larger than $n+2$.
\end{proof}

\begin{proposition} \label{tight1}
    The lower bound is tight for $n\ge 7$, meaning that $G$ has a labeled plane-saturated subgraph with $n+2$ edges.
\end{proposition}
\begin{proof}
    A saturated subgraph with $n+2$ edges can be seen on Figure \ref{fig:n+2}.
\end{proof}

Theorem \ref{sec2main} now clearly implies Theorem \ref{lower} if we just pick $G$ as our graph.

\section{Proof of the Upper Bound on $lpsr(G)$}\label{sec3}
In order to prove Theorem \ref{upper}, we use several lemmas. These could be useful in other arguments too, so we apply numbering on these lemmas unlike in the previous section. The number of vertices of $G$ is denoted by $n$ throughout the section.

\begin{lemma} \label{cycle}
    If $G$ is a maximal planar graph, and $v$ is a vertex of degree $k$, then the $k$ neighbouring vertices of $v$ form a cycle in some order in $G$.
\end{lemma}

\begin{proof}
    Consider a plane drawing of $G$. The neighbouring vertices of $v$ can be ordered by the clockwise order of the curves $vv_i$ coming out of point $v$. Let $v_1$ and $v_2$ be neighbours of $v$, next to each other in clockwise order. We prove that there is a Jordan curve between them that does not intersect any other curves, and has no common point with them. Since $v$ is a neighbour of both of them, there are Jordan curves from $v$ to each that do not intersect any other curve, and $v$ does not have an edge that goes between the two curves. Therefore we can just draw a new non-crossing curve from $v_1$ to $v_2$ by following these Jordan curves with an infinitesimally close curve. See Figure \ref{close} for intuition.
    Therefore there is an edge between consecutive vertices, so they form a cycle.
\end{proof}

\begin{figure}
	\begin{center}
		\includegraphics[scale=0.8]{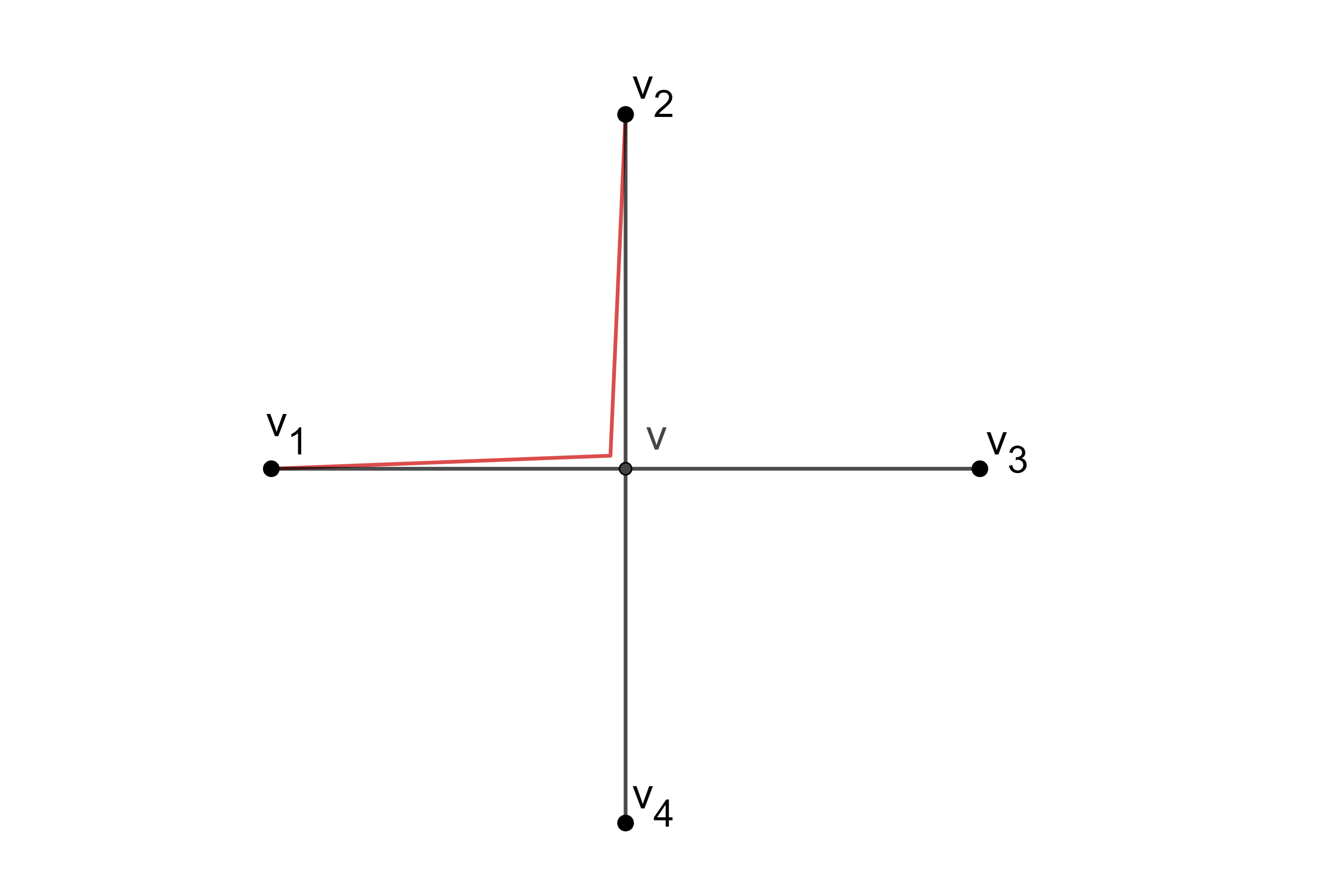}
		\caption{If $v_1$ and $v_2$ were not connected already, we could add an extra edge (with red), contradicting the maximality of $G$ as a planar graph.}
		\label{close}
	\end{center}
 \end{figure}

\begin{lemma}\label{no3pair}
In a maximal planar graph $G$ of at least $5$ vertices, there are no two degree $3$ vertices connected by an edge.
\end{lemma}

\begin{proof}
Suppose we had such a graph, and consider a plane drawing of it. Let $v_1,v_2$ have degree $3$, and $(v_1,v_2)\in E(G)$. Let the other neighbours of $v_1$ be $v_3,v_4$. By Lemma \ref{cycle}, $v_2v_3$ and $v_2v_4$ are both edges of the graph, and since $deg(v_2)=3$, these give all the three edges from $v_2$.
Now note that, new vertices cannot be added to a face where one of the vertices cannot accommodate more edges.  This happens because whatever structure the face has, the vertex with no more possible edges is indistinguishable from any other inner points of the bounding curves, hence it will end up somewhere as an extra point on the boundary of a new face.  This would give a face with at least $4$ vertices on its boundary, contradicting the maximal planarity of $G$. See Figure \ref{fig:indist} for clarity. Since in the drawing of $v_1,v_2,v_3,v_4$, every face contains either $v_1$ or $v_2$ on its boundary, we cannot add any new vertices to the graph. Hence $|V(G)|=4$ which is also a contradiction. Thus if $G$ has size at least $5$, it cannot have two degree $3$ vertices joined by an edge.
\end{proof}

\begin{figure}
	\begin{center}
		\includegraphics[scale=0.8]{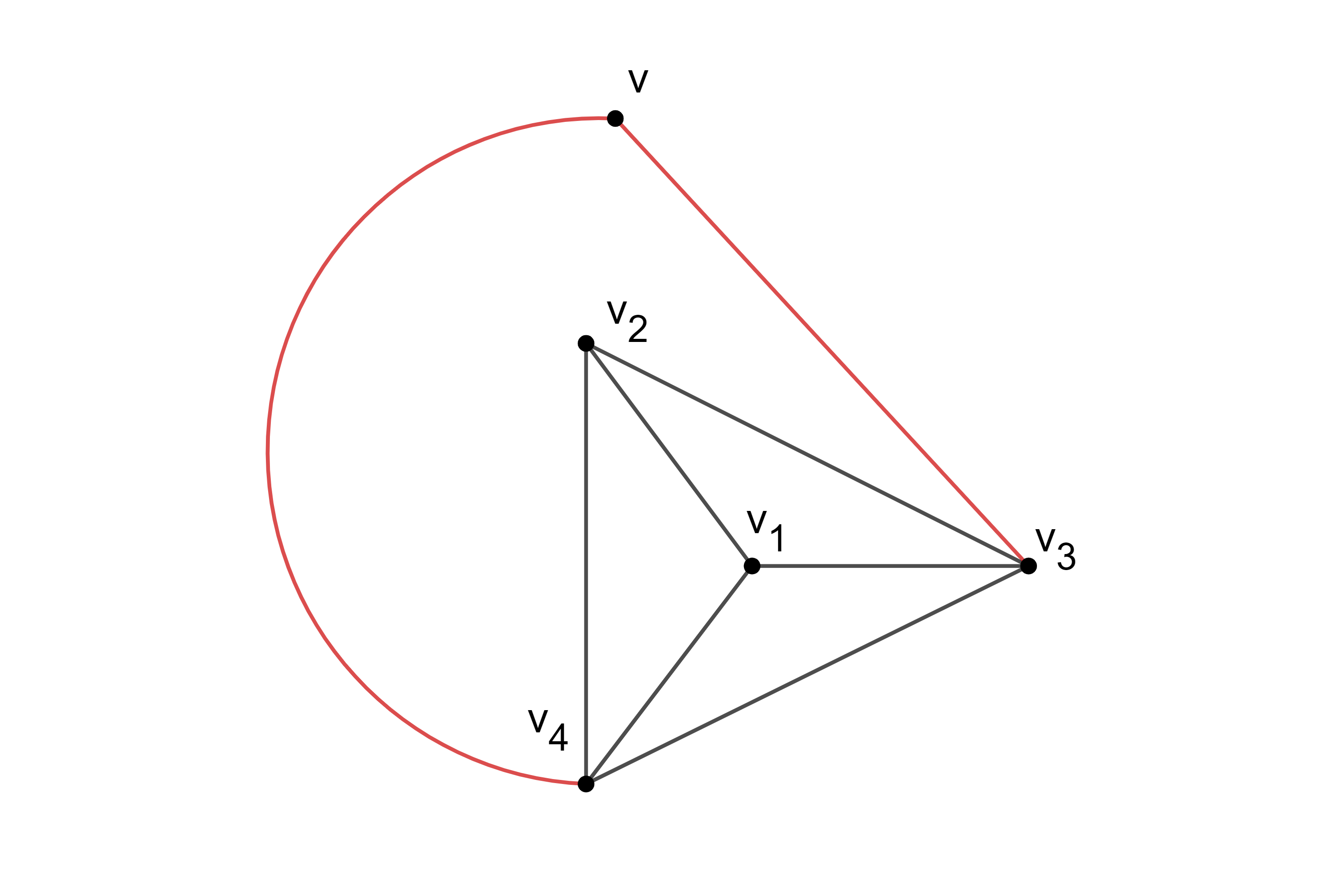}
		\caption{On the outside face, $v_2$ is indistinguishable from any other inner points of $v_3v_2$ and $v_2v_4$. Therefore if we draw anything new on the outer face, $v_2$ will be on the boundary of a face, where it functions as a non-vertex point, making this a non-triangular face. Note that since both $v_1,v_2$ are degree $3$, none of the four faces can have new vertices.}
		\label{fig:indist}
	\end{center}
 \end{figure}

\begin{lemma}\label{mainlemma1}
    Let $G$ be a maximal planar graph on $n$ vertices. %If $n\geq 26$, then there is a vertex of degree at least $4$ and at most $\frac n2$. 
    If $n\geq 47$, then there is a vertex of degree at least $4$ and less than $\frac n2$.
\end{lemma}

\begin{proof}

    Suppose otherwise that $G$ is a graph on $n\geq47$ vertices, and each of its vertices has degree $3$ or $\geq \frac n2$. Let it have $k_1$ vertices of degree $3$, and $k_2$ vertices of degree $\geq \frac n2$. We know that $k_1+k_2=n$.
    
    We also know from double counting the number of edges that $3k_1+\frac n2 k_2\leq 6n-12$. This implies $k_2\leq 11$. Then $k_1\geq n-11$. 
    Since $\frac n2 k_2\leq 6n-12-3k_1\leq3n+21$, we also know that $k_2\leq 6$.

    By Lemma \ref{no3pair}, all degree $3$ vertices are only adjacent to large degree vertices. By Kuratowski's theorem \cite{kuratowski1, kuratowski2}, the graph cannot contain a $K_{3,3}$. Therefore $G$ cannot have $3$ different degree $3$ vertices, connected to the same three large degree vertices. But since $k_2\leq6$, then if $k_1\geq 2\binom{6}{3}+1=41$, then by pigeonhole principle there will be $3$ such vertices, contradicting planarity of $G$. As $n\geq 47$ and $k_2\leq6$, $k_1\geq 41$ must hold. Therefore $G$ cannot have only degree $3$ or degree $\geq\frac n2$ vertices.
\end{proof}

\begin{lemma}\label{mainlemma2}
If $G$ has a degree $k\geq4$ vertex, then $G$ has a labeled plane-saturated subgraph $H$ with at most $\leq 2k+\frac6kn-\frac{12}{k}-4$ edges.
\end{lemma}

\begin{proof}
    Let $v$ be a degree $k$ vertex of $G$. As a start, draw the wheel formed by $v$, its neighbours $v_1, ...,v_k$, and the cycle on its neigbours. We can draw this wheel such that we decide which face contains each of the remaining labeled points, as the plane can be continuously transformed between these positions. The wheel contains $k\geq 4$ inner faces. We apply the $4$-color theorem \cite{4color1} to define a $4$-colouring $\chi$ on the vertices of $G$. We only care about the colours of the vertices in $G\verb|\|\{v, v_1, v_2, ..., v_k\}$ in the restriction of $\chi$, so let the four colour classes in this vertex set be $X_1, X_2, X_3, X_4$.

    Place these vertex sets in separate inner faces of the $k$-wheel already drawn. We can do that, since $k\geq 4$. This way, no further edges can be drawn between these $n-k-1$ vertices. We just need to count the maximum number of edges between a $v_i$ and some element of $X_j$. Note that we had a choice of how to put these classes in the faces; we consider all $k$ possible placement orders we can get by rotating the colour classes around $v$, by moving each colour class clockwise to the next region around $v$. See Figure \ref{placement} for clarity.
    Let us count the number of edges between some $v_i$ and $X_j$ altogether in these $k$ different placements. If we add them up, we count every single such edge twice. The total number of such edges is at most $3n-6-2k$, hence we counted at most $6n-12-4k$ edges. We choose the one rotation where this number of edges is minimal, hence we can choose one where we draw at most $\frac{6n-12-4k}{k}=\frac6kn-\frac{12}{k}-4$ new edges by pigeonhole principle. Together with the $k$-wheel, we have drawn $2k+\frac6kn-\frac{12}{k}-4$ edges, and the drawing is labeled plane-saturated. Hence there is a labeled plane-saturated subgraph with at most $ 2k+\frac6kn-\frac{12}{k}-4$ edges, which we wanted to prove.
    
\end{proof}

\begin{figure}
	\begin{center}
		\includegraphics[scale=0.8]{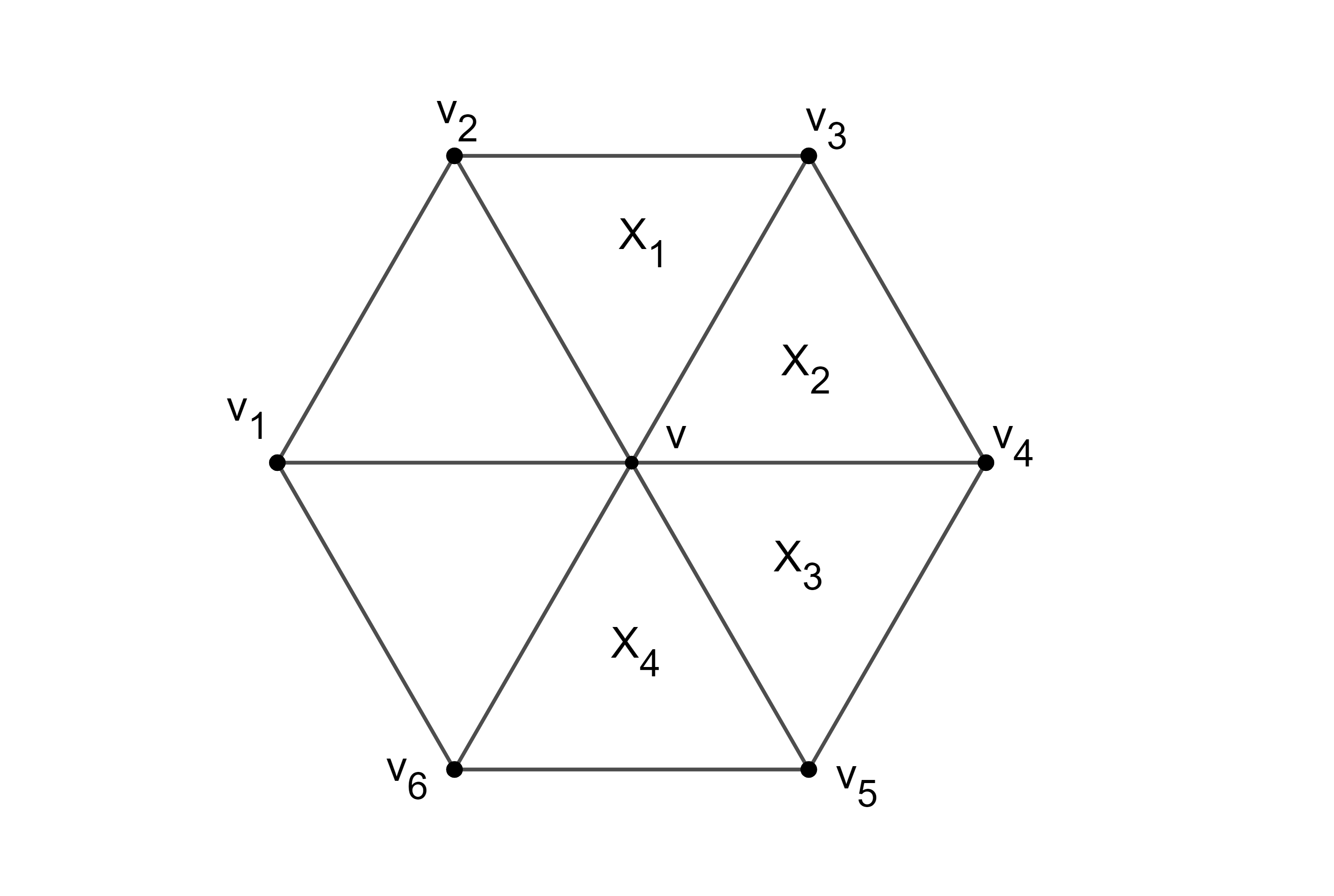}
		\caption{A 6-wheel around $v$, and one placement of the colour classes in the faces. We consider all $6$ possible placements of them in the faces we can get by a rotation around $v$, and take the one that induces the minimum number of edges between some $v_i$ and elements of some $X_j$.}
		\label{placement}
	\end{center}
 \end{figure}

 \begin{lemma}\label{k=4}
     If $G$ has a degree $4$ vertex, then it has a labeled plane-saturated subgraph with at most $n+5$ edges.
 \end{lemma}
\begin{proof}
    We use the same drawing as in the previous lemma, but give a more careful analysis. We show that in the graph $G$, we cannot have more than $2n-6$ edges going between elements of some $X_i$ and some vertex $v_j$ of the wheel. Each vertex of an $X_i$ can be connected to $0,1,2,3$ or $4$ different vertices $v_i$. If there are three vertices that are connected to the same three $v_i$'s, we get a contradiction by Kuratowski's theorem and by the planarity of $G$, as it would introduce a $K_{3,3}$. We now count edge triplets joining a vertex to $3$ different $v_i$'s. A vertex joined to all $4$ of them imposes $4$ triplets. Since there are $4$ outer vertices of the wheel, if the graph has at least $2\binom {4}{3}+1=9$ triplets, then by pigeonhole principle there will be three vertices joined to the same three $v_i$'s, causing a $K_{3,3}$. $v$ is already joined to all $4$ $v_i$'s, causing $4$ triplets. The coloured vertices (those in the $X_i$'s) can only cause at most $4$ more triplets, and the other coloured vertices have at most $2$ neighbours among $v_1,v_2,v_3,v_4$. There are more potential edges if these four triplets are caused by separate degree $3$ vertices instead of one degree $4$ vertex. This gives at most $4*3+(n-5-4)*2=2n-6$ edges between the wheel and the coloured vertices.

    Now we use the same calculation as in the previous lemma, but now double counting gives $4n-12$ edges instead of the $6n-12-4k$. Now we can choose a rotation which adds at most $n-3$ new edges to the wheel, therefore $G$ has a labeled plane-saturated subgraph with at most $n-3+8=n+5$ edges.
\end{proof}

\begin{lemma}\label{k=5}
     If $G$ has a degree $5$ vertex, then $f(G)\leq \frac{4}{5}n+\frac{46}{5}$.
 \end{lemma}

\begin{proof}
    The proof is very similar to that of the previous lemma. We again start with drawing the $5$-wheel.  We show that the graph $G$ cannot have more than $2n-2$ edges going between elements of some $X_i$ and some vertex $v_j$ of the wheel. Each vertex of an $X_i$ can be connected to $0,1,2,3,4$ or $5$ different vertices $v_i$. If there are $3$ vertices  that are connected to the same three $v_i$'s, we get a contradiction by Kuratowski's theorem and by the planarity of $G$. But since there are $5$ outer vertices of the wheel, if we have at least $2\binom {5}{3}+1=21$ triplets of $v_i$'s joined to the same vertex, then by pigeonhole principle there will be three vertices joined to the same three $v_i$'s, causing a contradiction. $v$ already causes $\binom{5}{3}=10$ triplets, so from the extra vertices, at most $10$ more triplets can be imposed. The most additional edges are possible if these triplets are caused by different degree $3$ vertices, and the others have at most $2$ neighbours on the wheel. This gives at most $10*3+(n-6-10)*2=2n-2$ edges between the wheel and the coloured vertices.

    Now we can use the same calculation as in the previous lemma, but now double counting gives $4n-4$ edges instead of the $6n-12-4k$. Thus we can choose a rotation which adds at most $\frac{4}{5}n-\frac{4}{5}$ new edges to the wheel, so $G$ has a labeled plane-saturated subgraph with at most $\frac{4}{5}n-\frac{4}{5}+10=\frac{4}{5}n+\frac{46}{5}$ edges. 

\end{proof}

\begin{lemma}\label{k=6}
     If $G$ has a degree $6$ vertex, then $f(G)\leq \frac{2}{3}n+14$.
 \end{lemma}

 \begin{proof}
     The proof method is the same as in the previous lemma. In this case, Kuratowski's theorem gives that $2\binom {6}{3}+1=41$ triplets force a contradiction. Considering that the vertex $v$ adds $20$ triplets, we can have at most $20$ coloured vertices with three neighbours among the $v_i$'s and this is the most possible edges we can add. This gives at most $20*3+(n-7-20)*2=2n+6$ edges between the wheel and the coloured vertices.

After double counting the edges as before, and choosing the rotation with the fewest number of edges, we get a labeled plane-saturated subgraph with at most $\frac{4}{6}n+\frac{12}{6}+12=\frac{2}{3}n+14$ edges.
 \end{proof}

Now we are ready to prove Theorem \ref{upper}.

\begin{proof}[Proof of Theorem \ref{upper}]
    By Lemma \ref{mainlemma1}, $G$ has vertex $v$ of degree at least $4$ and less than $\frac{n}{2}$. 
    \begin{itemize}
        \item If $deg(v)=4$, then by Lemma \ref{k=4} there is a labeled plane-saturated subgraph with at most $n+5$ edges. 
        \item If $deg(v)=5$, then by Lemma \ref{k=5} there is a labeled plane-saturated subgraph with at most $\frac{4}{5}n+\frac{46}{5}\leq n+5$ edges since $n\geq 47$.
        \item If $deg(v)=6$, then by Lemma \ref{k=6} there is a labeled plane-saturated subgraph with at most $\frac{2}{3}n+14\leq n+5$ edges since $n\geq 47$.
        \item If $7\leq deg(v)\leq \frac {n-1}{2}$ then by Lemma \ref{mainlemma2}, and considering that the bound in the lemma has maximal value either at $k=7$ or $k=\frac n2$, there is a labeled plane-saturated subgraph with at most  $max(14+\frac67n+\frac{12}{7}-4,n-1+\frac{12n}{n-1}+\frac{24}{n-1}-4)=max(\frac{6}{7}n+\frac{82}{7},n+7+\frac{36}{n-1})< n+8$ edges. 
    \end{itemize}

    Dividing these edge numbers by $e(G)=3n-6$, and using that the number of edges is an integer, we get the upper bound $lpsr(G)\leq \frac{n+7}{3n-6}$ for $n\geq47$.

Theorem \ref{upper} is now proven.
\end{proof}

\section{The unlabeled problem}\label{sec4}

In this section we move our focus to the plane-saturation ratio. We show Theorem \ref{lowerpsr}.

\begin{theorem}\label{const3n/2}
    For all $n$, there exists a maximal planar graph $G$ with $n$ vertices where any plane-saturated subgraph has at least $3n/2-3$ edges.
\end{theorem}
\begin{proof}
For $n\le 4$, it is trivial to check that the bound holds: if $n=1$, or $2$, a maximal planar graph is a tree so the only plane-saturated subgraph is the entire graph; if $n=3$ or $4$, a plane-saturated subgraph must include a cycle, and thus at least three edges.
    For $n\ge 5$, again consider the graph $G=(V,E)$ from Section 2, with $V=\{v_1,\dots,v_n\}$ and $E=\{(v_1,v_i)\mid 3\le i\le n\}\cup\{(v_2,v_i)\mid 3\le i\le n\}\cup\{(v_i,v_{i+1})\mid 3\le i\le n-1\}\cup \{(v_3,v_n)\}$.

    There are two possibilities for our plane-saturated subgraph $H$. Either there exists some embedding of $H$ into $G$ such that some $(n-2)$-cycle of $H$ corresponds to $v_3v_4\dots v_n$ or there doesn't. If some cycle corresponds to $v_3v_4\dots v_n$, then it means the cycle has no chords. Then, for each vertex on this cycle, we can add an edge from that vertex to one of the two vertices corresponding to $v_1, v_2$ without causing a crossing, unless that edge is already present. Thus, in order to be saturated, at least $n-2$  edges must be present besides those on the $(n-2)$-cycle, giving a total of at least $2n-4\ge 3n/2-3$ edges.

    Otherwise, consider an embedding of $H$ into $G$. The vertices corresponding to $v_1, v_2$ are such that every cycle of $H$ contains at least one of them. Every other vertex is either on a face or contained inside a face where one of these two vertices is on the boundary. As such, it is possible to add an edge between any other vertex and one of the vertices corresponding to $v_1, v_2$ without adding a crossing (unless that edge is already present). Thus, to be saturated, every other vertex must have at least one edge to one of the vertices corresponding to $v_1, v_2$. 
    
    We also want to keep track of the number of edges in the drawing between vertices corresponding to $v_3,\dots,v_n$. Each such vertex is incident to $0,1$, or $2$ such edges. Arbitrarily deciding which vertex of the drawing corresponds to $v_1$ and which to $v_2$, we define the counts $A_1,A_2,A_3,B_1,B_2,B_3,C_1,C_2$, and $C_3$ as follows. For $i=1,2$, $A_i$ is the number of vertices corresponding to $v_3,\dots,v_n$ that are adjacent only to the vertex corresponding to $v_i$ and thus to none of those corresponding to $v_3,\dots,v_n$. $A_3$ is the number of vertices corresponding to $v_3,\dots,v_n$ that are adjacent to both the vertices corresponding to $v_1,v_2$, but none of those corresponding to $v_3,\dots,v_n$. The $B_i$'s and $C_i$'s are defined analogously with the $B_i$'s counting vertices with one neighbour in $\{v_3,\dots,v_n\}$ and the $C_i$'s counting vertices with two neighbours in $\{v_3,\dots,v_n\}$.

    Thus, the total number of edges in the drawing is \begin{align*}
    &\frac{\sum_{i=1}^3 B_i+2\sum_{i=1}^3 C_i}{2}+(A_1+A_2+B_1+B_2+C_1+C_2)+2(A_3+B_3+C_3)\\
    &=3/2(\sum_{i=1}^3(A_i+B_i+C_i))+A_3/2+B_3+(C_1+C_2)/2+3C_3/2-(A_1+A_2)/2\\
    &=3(n-2)/2+A_3/2+B_3+(C_1+C_2)/2+3C_3/2-(A_1+A_2)/2.
    \end{align*}

    We now need to establish that $A_1+A_2$ is small enough relative to the other terms that we can still get a lower bound of $3(n-2)/2$.
    Consider the neighbours, $q_1,q_2,\dots,q_m, q_{m+1}:=q_1$ in clockwise order of the vertex corresponding to $v_1$. It is always possible to add an edge between $q_i$ and $q_{i+1}$ without adding a crossing. If $q_i$ is a vertex enumerated by $A_1$ (or $A_3$), then it has no neighbours among the vertices corresponding to $\{v_3,\dots,v_n\}$ which means edges $q_iq_{i+1}$ or $q_iq_{i-1}$ are not present in the drawing (remember the vertex corresponding to $v_2$ is not a neighbour of the vertex corresponding to $v_1$) and we are forbidden from adding them since the drawing is saturated. This means that $q_{i-1}, q_{i+1}$ both already have two neighbours among the vertices corresponding to  $\{v_3,\dots,v_n\}$. Thus as long as $m\ge 2$, the number of neighbours of the vertex corresponding to $v_1$ with two neighbours corresponding to  $\{v_3,\dots,v_n\}$ is at least as large as the number with zero neighbours corresponding to  $\{v_3,\dots,v_n\}$. This gives that $C_1+C_3\ge A_1+A_3$. A similar analysis on the vertex corresponding to $v_2$ and its neighbours gives $C_2+C_3\ge A_2+A_3$. Thus,
    \[
    (C_1+C_2)/2+C_3\ge (A_1+A_2)/2,
    \]

    so $3(n-2)/2+A_3/2+B_3+(C_1+C_2)/2+3C_3/2-(A_1+A_2)/2\ge 3(n-2)/2$, as desired.

    We have left unaddressed the case where the vertex corresponding to $v_1$ (similarly for $v_2$) has a single neighbour. However, this cannot happen, since the vertex corresponding to $v_1$ will be contained inside a face with at least two boundary vertices besides the vertex corresponding to $v_2$. It is then possible to add a new edge from at least one of these vertices to the vertex corresponding to $v_1$, without introducing a crossing. Thus, such a drawing is not saturated.
\end{proof}

Now the proof of Theorem \ref{lowerpsr} is immediate after dividing by $e(G)=3n-6$.

\begin{proposition}\label{construction 3n/2}
    For $n\ge 9$ and odd, the graph discussed in the proof of Theorem \ref{const3n/2} actually has a plane-saturated subgraph of size at most $3n/2+3/2$.
\end{proposition}
\begin{proof}
    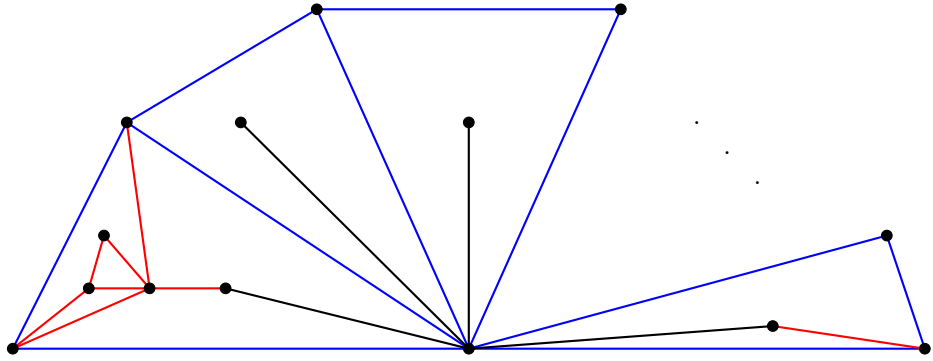
\begin{figure}
    \centering
    \begin{tikzpicture}
    [vertex/.style={circle, draw=black!100, fill=black!100, inner sep=0pt, minimum size=4pt},
    every edge/.style={draw, black!100, thick}]
    
    \draw [blue!100,thick] (0,0)--(-6,0) -- (-4.5,3)--(-2,4.5)--(2,4.5)--(0,0);
    \draw[blue!100,thick](0,0)--(5.5,1.5)--(6,0)--(0,0);
    \draw[blue!100,thick](-4.5,3)--(0,0)--(-2,4.5);

    \draw[black!100,thick](-3.2,0.8)--(0,0)--(-3,3);

    \draw[red!100,thick](-4.2,0.8)--(-3.2,0.8);

    \draw[black!100,thick](0,3)--(0,0)--(4,0.3);
    
    \draw[red!100, thick](4,0.3)--(6,0);

     \draw[red!100,thick](-6,0)--(-4.2,0.8)--(-4.5,3);

     \draw[red!100,thick](-6,0)--(-5,0.8)--(-4.8,1.5)--(-4.2,0.8)--(-5,0.8);

    \node [vertex] at (0,0) {};
     \node [vertex] at (-4.2,0.8) {};

 \node [vertex] at (-5,0.8) {};
  \node [vertex] at (-4.8,1.5) {};
     
    \node [vertex] at (-6,0) {};

    \node [vertex] at (-4.5,3) {};
    \node [vertex] at (-2,4.5) {};
    
    \node [vertex] at (2,4.5) {};
    \node  [vertex] at (5.5,1.5) {};
    \node  [vertex] at (6,0) {};

    \node [vertex] at (-3.2,0.8) {};
    \node [vertex] at (-3,3) {};
    
    \node [vertex] at (0,3) {};
    \node [vertex] at (4,0.3) {};

\node [] at (3,3) {.};
\node [] at (3.4,2.6) {.};
\node [] at (3.8,2.2) {.};

    \end{tikzpicture}    
    \caption{Saturated drawing for Proposition \ref{construction 3n/2}.}
    
    \label{fig:3n/2}
\end{figure}
Consider the following plane subgraph $H$ as shown in Figure~\ref{fig:3n/2}: Beginning with a star on $n-3$ vertices, draw edges from the first leaf to the third leaf, then the third leaf to the fifth leaf, then the fifth leaf to the seventh leaf, continuing all the way to the last leaf to form a path with $(n-5)/2$ edges, where each leaf not part of the new path ends up enclosed in a triangular face formed by the center of the star and the two leaves on either side. Add an additional edge between the last two leaves of the original star. Then add a vertex inside the face formed by the center vertex of the original star and the first and third leaves. Add edges from this new vertex to the first three leaves of the original star. Within the new face formed by this vertex and the first and third leaves of the original star, add another vertex adjacent to the new vertex and the first leaf. Then within the new face formed by the newest two vertices and the first and third leaves, add another new vertex adjacent to the newest two vertices.

    Because $G$ has just two vertices of degree at least $5$, the vertices of $H$ corresponding to $v_1$ and $v_2$ are determined up to swapping. Of the vertices corresponding to $v_3,\dots, v_n$, the only ones which do not already have two neighbours also corresponding to $v_3,\dots,v_n$ cannot have any edges added among them without causing a crossing, so indeed the drawing is saturated. 
    
    The number of blue triangles in Figure \ref{fig:3n/2} is $(n-5)/2$. These triangles contribute $n-4$ edges. There are then an additional $(n-5)/2$ edges incident to the highest degree vertex. There are five additional edges incident to the next highest degree vertex of degree $5$, and three more previously unaccounted for edges for a total of 
    $$(n-4)+(n-5)/2+5+3=3n/2+3/2$$ edges.
\end{proof}

Now we prove Theorem \ref{psrupper}. It is the direct corollary of two lemmas:

\begin{lemma}\label{lem:dorn-dsave}
If a maximal planar graph $G$ on $n$ vertices has maximum degree $d$, then $G$ has a plane-saturated subgraph with at most $3n-6-min(d,n-d-1)$ edges.
\end{lemma}

\begin{proof}
    Draw a wheel where the central vertex $v$ has degree $d$. This determines that in the embedding, $v$ must be mapped to one of the maximum degree vertices of $G$. Then we equidistribute the remaining vertices between the triangular faces of the wheel.

    In each face that contains at least $1$ vertex inside, there will be a non-triangular face, since $v$ has no edges going inside the triangle. So with each face that contains a vertex, we remove at least $1$ edge from the maximal drawing of $G$. If the max degree is $d$, we have exactly $min(d,n-d-1)$ such faces, proving the theorem.
\end{proof}

If $G$ only has very small degree vertices, or it has vertices of degree almost $n$, the above theorem is not too helpful; for in-between cases it gives a better bound. For example when $d=cn$ for some $0<c<1$ it can give a linear improvement compared to $3n-6$. 

    \begin{lemma}\label{lem:C2save}
Let $0<C_1\leq1$ be a constant, and let $C_2= \frac{C_1^2}{12+C_1}$. Let $n\geq \frac{5}{C_2}$ be an integer. If $G$ is a maximal planar graph on $n$ vertices, and it has a vertex $v$ of degree $deg(v)\geq C_1n$, then $G$ has a plane-saturated subgraph with at most $(3-C_2)n+1$ edges.
\end{lemma}
\textbf{Proof:} We demonstrate that if $G$ contains such a vertex $v$, then there exists a vertex $w$ in $G$ such that for any other vertex $w'$ with $\deg(w') < \deg(w)$, the inequality  
$$
\deg(w) - \deg(w') > C_2n
$$  
holds. Suppose, for the sake of contradiction, that no such vertex $w$ exists. Then, we can construct a sequence of vertices $v = v_0, v_1, v_2, \ldots, v_N$ such that  
$$
\deg(v_{i+1}) \geq \deg(v_i) - C_2n \quad \text{for all } i.
$$  
 Additionally, for sufficiently large $n$, we can choose $N$ so that
$$
\deg(v_N) \leq C_2n.
$$  as the graph must have some low degree (at most $5$) vertices.

If we add the number of edges these vertices have, it is upper bounded by $6n-12$, the degree sum of $G$.
\\
$$6n-12\geq\sum_{i=0}^Ndeg(v_i)\geq\sum_{i=0}^Nmax(0,C_1n-iC_2n)\geq \sum_{i=1}^{\lfloor \frac{C_1}{C_2} \rfloor} iC_2n=C_2n\frac{\lfloor \frac{C_1}{C_2} \rfloor(\lfloor \frac{C_1}{C_2} \rfloor+1)}{2}\geq \frac{(\frac{C_1}{C_2}-1)C_1}{2}n.$$

    If $\frac{(\frac{C_1}{C_2}-1)C_1}{2}\geq 6$, then the previous inequality fails. This happens if $$C_2\leq \frac{C_1^2}{12+C_1}$$ which is precisely the case, yielding the desired contradiction.

    Now for the upper bound we use the vertex $w$ whose existence we just proved.

    We draw a vertex with $deg(w)-C_2n$ vertices as neighbours, and we also draw the corresponding part of the cycle around the vertex, that is a path connecting these neighbours in order. By what we just proved, that vertex must have at least $C_2n$ other neighbours in any embedding of $H$ into $G$. 
    Now note that if we pick $2$ vertices next to each other (in a plane drawing of $G$) among the neighbours of $w$, the edge between them lies on two triangles. The one that does not contain $w$ can be drawn in such a way that it contains the existing drawing within its interior. Then just place all the remaining vertices outside of this new triangle, which prevents at least $C_2n-1$ edges from being present in $H$ (the edges between the vertex in $G$ corresponding to the middle vertex and its remaining neighbours, except possibly to $p$).
    See Figure \ref{fig:upper2} for an illustration where the third vertex, $p$, of this triangle is not a neighbour of the central vertex. It is also possible that $p$ is a neighbour of the central vertex and thus already present in the drawing. In that case, drawing just one of the edges from one of the consecutive neighbours to $p$ is sufficient to cut off the central vertex from the outer face.

    Note that the last two lemmas directly imply Theorem \ref{psrupper}. The value $d=\frac{\sqrt{217}-11}{4}n\approx 0.933n$ is precisely the maximum degree where the two bounds have equal strength. For larger $d$, Lemma~\ref{lem:C2save} yields a better bound assuming $n\geq75\geq\frac{5}{C_2}$ from the extremal case, while for smaller $d$, Lemma~\ref{lem:dorn-dsave} does.

    \begin{figure}
    \centering
    \includegraphics[scale=0.5]{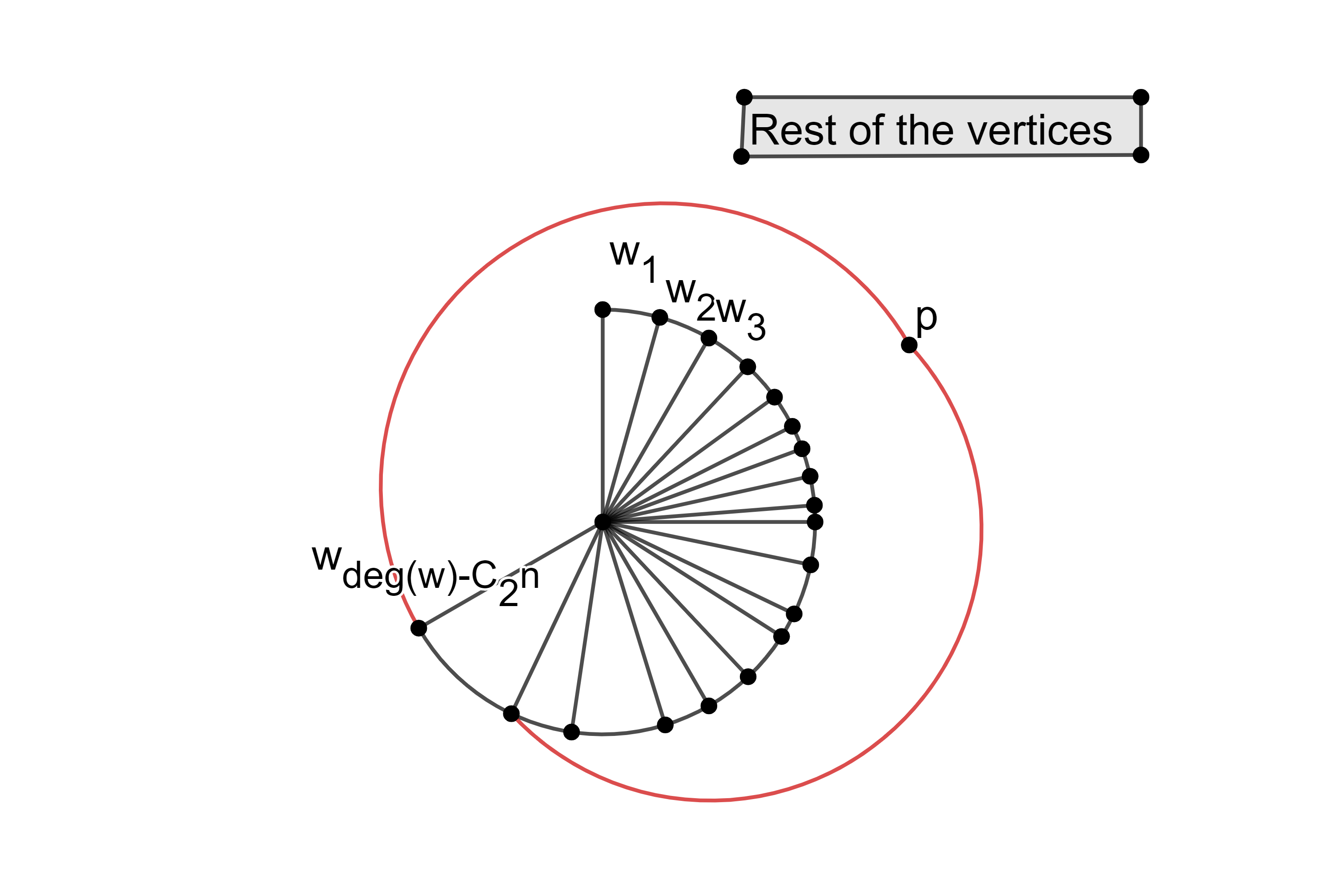}
    \caption{There must be another path of 2 edges between the last two vertices of the drawing, so either we can join the last vertex back to the cycle with an edge, or we can do the same with one extra vertex. We save at least $C_2n$ edges with this drawing since vertex of $G$ corresponding to the middle vertex has degree at least $deg(w)$.}
    \label{fig:upper2}
\end{figure}

\section{Concluding Remarks}\label{sec5}

 For the class of maximal planar graphs, we essentially showed the maximum possible value of the labeled plane-saturation ratio. This is not the case for the plane-saturation ratio, where we only managed to show a lower bound for the maximum. We suspect that this lower bound is asymptotically optimal:
 \begin{conj}
     If $n$ tends to infinity, the maximum value of the plane-saturation ratio over the set of maximal planar graphs $G$ on $n$ vertices tends to $\frac12$.
 \end{conj}

 In fact, any upper bound that is linearly better than $3n-6$ for the number of edges of the smallest plane-saturated subgraph of a maximal planar graph on $n$ vertices would be an interesting result. We could only show such an improvement for a very special class of graphs with maximum degree equal to $cn$ for some $0<c<1$.

 Plane-saturation can also be considered for other classes of graphs besides maximal planar. For example in \cite{clifton} the authors investigated the so called twin-free graphs.

 \begin{problem}
 Investigate the plane-saturation ratio for other special classes of graphs.
 \end{problem}

 One can also try to extend the concept of planar graphs to $k$-planar graphs. A $k$-planar graph is a graph, that can be drawn on the plane in a way that each one of its edges has at most $k$ crossings on it. One can define $k$-plane-saturation on these graph classes. As a start, maybe $1$-planar maximal graphs could be interesting.

\begin{problem}
    Investigate the maximum value of the (labeled) $1$-plane-saturation ratio for maximal $1$-planar graphs.
\end{problem}


\begin{thebibliography}{1}

\bibitem{4color1}
\textsc{K. Appel, W. Haken}, Every planar map is four colorable, \textit{A.M.S. Contemp. Math.
98 (1989).}

\bibitem{barat}
\textsc{J. Barát, Z. L. Blázsik, B. Keszegh, D. T. Nagy, Z. Zheng}, Plane saturation, \textit{Ongoing project}. 

\bibitem{brass}
\textsc{P. Brass, W. O. J. Moser, J. Pach}, Research problems in discrete geometry, \textit{Vol. 18, Springer, 2005.} \\\href{https://doi.org/10.1007/0-387-29929-7}{\tt\small 
https://doi.org/10.1007/0-387-29929-7}

\bibitem{cameron}
\textsc{A. Cameron, G. J. Puleo}, A lower bound on the saturation number, and graphs for which it is sharp, \textit{Discrete Mathematics 345 (2022), no. 7, 112867.} \\\href{https://doi.org/10.1016/j.disc.2022.112867}{\tt\small 
https://doi.org/10.1016/j.disc.2022.112867}

\bibitem{chen1}
\textsc{Y-C. Chen}, Minimum $C_5$-saturated graphs, \textit{J. Graph Theory, 61(2) (2009), 111–
126.} \\\href{https://doi.org/10.1002/jgt.20372}{\tt\small 
https://doi.org/10.1002/jgt.20372}

\bibitem{chen2},
\textsc{Y-C. Chen} All minimum $C_5$-saturated graphs, \textit{J. Graph Theory, 67(1) (2011), 9–26.} \\\href{https://doi.org/10.1002/jgt.20508}{\tt\small 
https://doi.org/10.1002/jgt.20508}

\bibitem{clifton}
\textsc{A. Clifton, N. Salia}, Saturated partial embeddings of planar graphs, \textit{arXiv, 2024} .\\\href{https://doi.org/10.48550/arXiv.2403.02458}{\tt\small 
https://doi.org/10.48550/arXiv.2403.02458}

\bibitem{erdos}
\textsc{P. Erdős, A. Hajnal, J. W. Moon,} A problem in graph theory, \textit{Amer. Math. Monthly, 71 (1964),1107–1110.}

\bibitem{faudree}
\textsc{J. R. Faudree, R. J. Faudree, J. R. Schmitt}, A survey of minimum saturated graphs, \textit{The Electronic
Journal of Combinatorics 1000 (2011), DS19–Jul.} \\\href{https://doi.org/10.37236/41}{\tt\small 
https://doi.org/10.37236/41}

\bibitem{faudree2}
\textsc{R. J. Faudree, R. J. Gould}, Saturation numbers for nearly complete graphs, \textit{Graphs and Combinatorics
29 (2013), no. 3, 429–448.} \\\href{https://doi.org/10.1007/s00373-011-1128-9}{\tt\small 
https://doi.org/10.1007/s00373-011-1128-9}

\bibitem{fisher}
\textsc{D. C. Fisher, K. Fraughnaugh, L. Langley}, On $C_5$-saturated
graphs with minimum size, In \textit{Proceedings of the Twenty-sixth Southeastern
International Conference on Combinatorics, Graph Theory and Computing
(Boca Raton, FL, 1995), volume 112, pages 45–48, 1995.}

\bibitem{fulek}
\textsc{R. Fulek, A. J. Ruiz-Vargas}, Topological graphs: empty triangles and disjoint matchings, \textit{Proceedings of the 29th Annual Symposium on Computational Geometry, 2013, pp. 259–266.} \\\href{https://doi.org/10.1145/2462356.2462394}{\tt\small 
https://doi.org/10.1145/2462356.2462394}

\bibitem{furedi}
\textsc{Z. Füredi, Y. Kim}, Cycle-saturated graphs with minimum number of edges, \textit{J. Graph Theory, 73 (2013), 203-215. } \\\href{https://doi.org/10.1002/jgt.21668}{\tt\small 
https://doi.org/10.1002/jgt.21668}

\bibitem{kuratowski1}
\textsc{C. Kuratowski}, Sur le problème des courbes gauches en Topologie, \textit{Fundamenta Mathematicae 15 (1930), 271-283.} \\\href{https://doi.org/10.4064%2Ffm-15-1-271-283}{\tt\small 
https://doi.org/10.4064\%2Ffm-15-1-271-283}

\bibitem{kyncl}
\textsc{J. Kynčl, J. Pach, R. Radoičič, G. Tóth}, Saturated simple and k-simple topological graphs, \textit{Computational Geometry: Theory and Applications 48 (2015), 295–310.} \\\href{https://doi.org/10.1016/j.comgeo.2014.10.008}{\tt\small 
https://doi.org/10.1016/j.comgeo.2014.10.008}

\bibitem{ollmann}
\textsc{L. T. Ollmann}, $K_{2,2}$-saturated graphs with a minimal number of edges, In \textit{Proceedings of the Third Southeastern Conference on Combinatorics, Graph
Theory, and Computing (Florida Atlantic Univ., Boca Raton, Fla., 1972),
pages 367–392, Boca Raton, Fla., 1972. Florida Atlantic Univ.} 

\bibitem{pach}
\textsc{J. Pach, J. Solymosi, G. Tóth}, Unavoidable configurations in complete topological graphs, \textit{Discrete
and Computational Geometry 30 (2003), 311–320.} \\\href{https://doi.org/10.1007/s00454-003-0012-9}{\tt\small 
https://doi.org/10.1007/s00454-003-0012-9}

\bibitem{suk}
\textsc{A. Suk}, Disjoint edges in complete topological graphs, \textit{Proceedings of the 28th Annual Symposium on
Computational Geometry, 2012, pp. 383–386.} \\\href{https://doi.org/10.48550/arXiv.1110.5684}{\tt\small 
https://doi.org/10.48550/arXiv.1110.5684}

\bibitem{kuratowski2}
\textsc{C. Thomassen}, Kuratowski's theorem, \textit{J. Graph Theory, (1981) 5: 225-241.} \\\href{https://doi.org/10.1002/jgt.3190050304}{\tt\small 
https://doi.org/10.1002/jgt.3190050304}

\bibitem{zykov}
\textsc{A. A. Zykov}, On some properties of linear complexes, \textit{Matematicheskii Sbornik 66 (1949),
no. 2, 163–188.}

\end{thebibliography}
\end{document}